\documentclass[12pt]{amsart}

\usepackage{amssymb}
\usepackage{amsmath}
\usepackage{hyperref}
\usepackage[pdftex]{graphicx}
\newtheorem{theorem}{Theorem}[section]
\newtheorem{lemma}[theorem]{Lemma}
\newtheorem{prop}[theorem]{Proposition}
\newtheorem{cor}[theorem]{Corollary}
\theoremstyle{definition}

\newcommand{\Mcg}{\mathrm{Mod}}

\newcommand{\C}{\mathcal{C}}
\newcommand{\Cc}{\mathcal{C}}
\newcommand{\D}{\mathcal{D}}

\newcommand{\cY}{\mathcal{Y}}

\newcommand{\T}{\mathcal{T}}

\newcommand{\Lk}{\mathrm{Lk}}
\numberwithin{equation}{section}
\title[Curve complex of the 3-holed projective plane]
{A note on the curve complex of the 3-holed projective plane}
\author{B{\l}a\.zej Szepietowski}
\email{blaszep@mat.ug.edu.pl}
\address{Institute of Mathematics, Faculty of Mathematics, Physics and Informatics, University of Gda\'nsk, 80-308 Gda\'nsk, Poland} 
\thanks{Supported by grant 2015/17/B/ST1/03235 of  National Science Centre, Poland.}
\begin{document}
\begin{abstract}
Let $S$ be a projective plane with $3$ holes. We prove that there is an exhaustion of the curve complex $\C(S)$ by a sequence of finite rigid sets. As a corollary, we obtain that the group of simplicial automorphisms of $\C(S)$ is isomorphic to the mapping class group $\Mcg(S)$. We also prove that $\C(S)$ is quasi-isometric to a simplicial tree.
\end{abstract}

\maketitle
\section{introduction}
The complex of curves $\Cc(S)$ of a surface $S$, first introduced by Harvey \cite{Harvey}, is
the simplicial complex with $k$-simplices representing collections
of homotopy classes of $k+1$ non-isotopic disjoint simple closed
curves in $S$. 
In this paper we let $S=N_{1,3}$ be the $3$-holed projective plane. Then $\C(S)$ is one-dimensional and its combinatorial structure was described by Scharlemann \cite{Sch}. 
The first purpose of this note is to prove some rigidity results about $\C(N_{1,3})$, which are known for most surfaces, but have not been proved in the literature in this particular case. The second purpose is to show that $\C(N_{1,3})$ is quasi-isometric to a simplicial tree.

%Let us describe the background for the first part.
%There is a natural action of the mapping class
%group $\Mcg(S)$ on $\Cc(S)$, which plays a very important role in
%the study of the group $\Mcg(S)$.
By the celebrated theorem of Ivanov \cite{Ivanov-Aut}, Korkmaz \cite{Kork-CC} and Luo \cite{Luo}, the group $\mathrm{Aut}(\Cc(S))$ of simplicial
  automorphisms of $\Cc(S)$ for orientable surface $S$ is, with a few well understood exceptions, isomorphic to the extended mapping class group $\Mcg^\pm(S)$.
A stronger version of this result, due to Shackleton \cite {Shack}, says that every locally injective simplicial map from $\Cc(S)$ to itself is induced by some element of $\Mcg^\pm(S)$ (simplicial map is locally injective
if its restriction to the star of every vertex is injective).  Analogous results for 
nonorientable surfaces were proved by Atalan-Korkmaz \cite{AK} and Irmak \cite{Irmak14}, omitting the case of $N_{1,3}$.

Aramayona and Laininger introduced in \cite{AL1} the notion of a {\it rigid set}.
It is a subcomplex $X\subset\Cc(S)$, with the property that every
locally injective simplicial map $X\to\Cc(S)$ is induced by some
homeomorphism of $S$.
They constructed in \cite{AL1} a finite rigid
set in $\Cc(S)$, for every orientable surface $S$, and in 
\cite{AL2} they proved that there is an exhaustion of $\Cc(S)$ by a sequence of finite rigid sets.

Let $S=N_{g,n}$ be a nonorientable surface of genus $g$ with $n$ holes.
Ilbira and Korkmaz \cite{IK} constructed a finite rigid set in $\Cc(S)$ for
$g+n\ne 4$. Irmak \cite{Irmak19} proved that $\C(S)$ can be exhausted by a sequence of finite rigid sets for $g+n\ge 5$ or $(g,n)=(3,0)$.
In this paper we show that the main results of Ilbira-Korkmaz \cite{IK} and 
Irmak \cite{Irmak19} are true also for $N_{1,3}$.  
\begin{theorem}\label{T:main}
There exists a sequence 
$\cY_1\subset\cY_2\subset\cdots\subset\C(N_{1,3})$ such that:
\begin{itemize}
\item[(1)] $\cY_i$ is a finite rigid set for all $i\ge 1$;
\item[(2)] $\cY_i$ has trivial pointwise stabilizer in $\Mcg(N_{1,3})$ for all $i\ge 1$;
\item[(3)] $\bigcup_{i\ge 1}\cY_i=\C(N_{1,3})$.
\end{itemize}
\end{theorem}
Our proof is %based on the description of  $\C(N_{1,n})$ due to Scharlemann \cite{Sch} and is 
independent from \cite{IK, Irmak19}.
The following corollary is an extension of the main results of Atalan-Korkmaz \cite{AK} and Irmak \cite{Irmak14}. It follows easily from Theorem \ref{T:main} (see the proof of the analogous corollary in \cite{AL2}).
\begin{cor}\label{C:main}
If $\phi\colon\C(N_{1,3})\to\C(N_{1,3})$ is a locally injective simplicial map, then there exists a unique $f\in\Mcg(N_{1,3})$ such that $\phi=f$.
\end{cor}
In particular, the group of simplicial automorphisms of $\C(N_{1,3})$ is isomorphic to $\Mcg(N_{1,3})$. 
%Corollary \ref{C:main} implies that the main results of Atalan-Korkmaz \cite{AK} and Irmak \cite{Irmak14} are true also for $N=N_{1,3}$, which was one of the cases left open by \cite{AK,Irmak14}.
%The second purpose of this note is to prove that $\C(N_{1,3}S)$ is quasi-isometric to a simplicial tree.

%
\section{Preliminaries.}
Let $S$ be a surface of finite type. By a {\it hole} in a surface we mean a boundary component. % or hole, and when the distinction is irrelevant, we will not distinguish between a boundary component and a hole.
A {\it curve} on $S$ is an embedded simple closed curve.  A curve is one-sided (resp. two-sided) if its regular neighbourhood is a M\"obius band (resp. an annulus).
If $\alpha$ is a curve on $S$, then $S\backslash\alpha$ is the subsurface obtained by removing from $S$ an open regular neighbourhood of $\alpha$.
A curve $\alpha$ is {\it essential} if no boundary component of $S\backslash\alpha$ is a disc or an annulus or a M\"obius band. 

The {\it curve complex} $\C(S)$ is a simplicial complex whose $k$-simplices correspond to sets of $k+1$ isotopy classes of essential curves on $S$ with pairwise disjoint representatives. To simplify the notation, we will confuse a curve with its isotopy class and the corresponding vertex of $\C(S)$. Simplices of dimension $1$, $2$ and $3$ will be called {\it edges}, {\it triangles} and {\it tetrahedra} respectively. 

For $\alpha,\beta\in\C^0(S)$ we denote by $i(\alpha,\beta)$ their geometric intersection number.

The {\it mapping class group} $\Mcg(S)$ of a nonorientable surface $S$ (resp. the {\it extendend mapping class group} $\Mcg^\pm(S)$ of an orientable surface $S$) is the group of isotopy classes of all self-homeomorphisms of $S$. % preserving the set of holes. 
If $S$ is orientable, then the mapping class group $\Mcg(S)$ is  defined to be the group if isotopy classes of orientation preserving homeomorphisms. Note that $\Mcg(S)$ and $\Mcg^\pm(S)$ act on $\C(S)$ by simplicial automorphisms.

If $S$ is a four-holed sphere, then $\C(S)$ is a countable set of vertices. In order to obtain a connected complex, the definition of $\C(S)$ is modified by declaring $\alpha,\beta\in\C^0(S)$ to be adjacent in $\C(S)$ whenever $i(\alpha,\beta)=2$. Furthermore, triangles are added to make $\C(S)$ into a flag complex. The complex $\C(S)$ so obtained is isomorphic to the well-known {\it Farey complex}. Two adjacent vertices of $\C(S)$ will be called {\it Farey neighbours}, and   $2$-simplices of $\C(S)$ will be called {\it Farey triangles}. If $C$ is a boundary component, then we denote by $\Mcg(S,C)$ (resp. $\Mcg^\pm(S,C)$) the subgroup of $\Mcg(S)$ (resp. $\Mcg^\pm(S)$) consisting of elements fixing $C$.
\begin{lemma}\label{L:Farey_rigid}
Suppose that $S$ is a $4$-holed sphere, and $C$ is one of the boundary components.
For any two Farey triangles $\{\beta_0,\beta_1,\beta_2\}$ and $\{\beta'_0,\beta'_1,\beta'_2\}$ in $\C(S)$ there exists a unique $f\in\Mcg^\pm(S,C)$ such that $f(\beta_i)=\beta'_i$ for $i=0,1,2$.
\end{lemma}
\begin{proof} 
By cutting $S$ along Farey neighbours we obtain four annuli, each containing one boundary component of $S$. Therefore there exists an orientation preserving $f'\in\Mcg(S,C)$ such that $f'(\beta_i)=\beta'_i$ for $i=1,2$. Furthermore, since $f'(C)=C$, such $f'$ is easily shown to be unique by the Alaxander method \cite[Prop. 2.8]{FM}.
 The pointwise stabilizer of $\{\beta'_1,\beta'_2\}$ in $\Mcg^{\pm}(S,C)$ is a cyclic group of order $2$ generated by an orientation reversing involution $\tau$ fixing every hole and such that $\beta'_0$ and $\tau(\beta'_0)$ are the unique common Farey neighbours of both $\beta'_1$ and $\beta'_2$. By composing $f'$ with $\tau$ if necessary we obtain the desired $f$.
\end{proof}

%Suppose that $S$ is a non-orientable surface.
%A subcomplex $Y$ of $\C(S)$ is {\it rigid} if for every locally injective simplicial map $\phi\colon Y\to\C(S)$ there exists $f\in\Mcg(S)$ such that
%$f_{\mid_Y}=\phi$.

We represent the surface $N_{1,n}$ as a sphere with one crosscap and $n$ holes.
The following two lemmas are easy to prove, and otherwise, they can be found in \cite{Sch}.
\begin{lemma}\label{Lem:2holedRP2}
$\C(N_{1,2})$ consists of two one-sided vertices $\alpha,\alpha'$ such that $i(\alpha,\alpha')=1$ (Figure \ref{Fig:N123}). 
\end{lemma}

\begin{lemma}\label{Lem:3holedRP2}
In $\C(N_{1,3})$ every two-sided vertex $\beta$ is connected by an edge with exactly two vertices $\alpha, \alpha'$ which are one-sided and $i(\alpha,\alpha')=1$. Conversely, for every pair of one-sided vertices $\alpha, \alpha'$ such that $i(\alpha,\alpha')=1$, there exists exactly one two-sided vertex $\beta$ connected by an edge with $\alpha$ and $\alpha'$
(Figure \ref{Fig:N123}).
\end{lemma}
\begin{figure}[hbt]
 \begin{center}
 \includegraphics{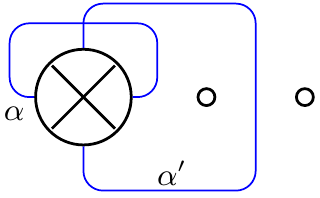} \hskip 1cm \includegraphics{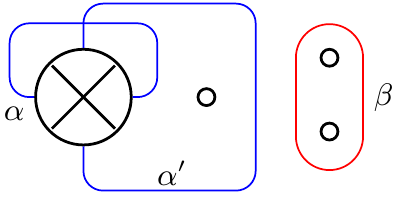}
\caption {Vertices of $\C(N_{1,2})$ (left) and $\C(N_{1,3})$ (right).}\label{Fig:N123}
\end{center}
\end{figure}

%The following proposition is well known and easy to prove.
%\begin{prop}\label{Prop:Farey}
%\begin{itemize}
%\item[(a)] Every edge of $\C(\Sigma)$ is contained in exactly two triangles.
%\item[(b)] $\PM(\Sigma)$ acts simply transitively on edges of $\C(\Sigma)$.
%\item[(c)] The stabilizer in $\PM^\pm(\Sigma)$ of and edge of  $\C(\Sigma)$ is a cyclic group of order $2$ generated by an orientation reversing involution which swaps the two triangles containing the given edge.
%\item[(d)]  $\PM^\pm(\Sigma)$ acts simply transitively on triangles  of $\C(\Sigma)$.
%\end{itemize}
%\end{prop}

%\begin{figure}
%\begin{tabular}{cc}
%\input{fig2}&\input{fig1}
%\end{tabular}
%\caption{\label{fig}Curves on the three-punctured projective plane $S$.}
%\end{figure}
\section{Finite rigid sets}
In this section $S$ denotes the three-holed projective plane. The complex $\C(S)$ was studied by Scharlemann \cite{Sch}. It is a bipartite graph: its vertex set can be partitioned as
$\C^0(S)=V_1\sqcup V_2$, where $V_1$ and $V_2$ denote the sets of one-sided and two-sided vertices respectively, and every edge of $\C(S)$ connects a one-sided vertex with a two-sided one. Furthermore, by Lemma \ref{Lem:3holedRP2}, every $\beta\in V_2$ is connected by an edge with exactly two $\alpha,\alpha'\in V_1$ such that $i(\alpha,\alpha')=1$. %Conversely, given $\alpha,\alpha'\in V_1$, $\beta$ with $i(\alpha,\beta)=1$, there exits exactly one two-sided curve $\gamma$ which is connected by an edge in $\C(S)$ to $\alpha$ and $\beta$. Namely, $\gamma$ is the boundary of a regular neighbourhood of $\alpha \cup\beta$. 
We say that $\beta$ is {\it determined} by $\alpha$ and $\alpha'$.

We define an auxiliary  simplicial complex $\D$ whose vertex set is $V_1$, and a set of vertices $\{\alpha_0,\dots,\alpha_k\}$ is a simplex if 
$i(a_i,a_j)=1$ for $0\le i< j\le k$. It follows from the above discussion that $\C(S)$ is isomorphic to the graph obtained by subdividing every edge of $\D^1$ --  the $1$-skeleton of $\D$. Indeed, the subdivision of an edge of $\D^1$ corresponds to adding the two-sided vertex determined by this edge.  

\begin{prop}\label{Prop:D}% $\mathcal{D}$ has the following properties:
\begin{itemize}
\item[(a)] The link of each vertex of $\D$  is isomorphic to the Farey complex.
\item[(b)] $\dim\D=3$
\item[(c)]  every triangle of $\D$ is contained in exactly two different tetrahedra.
\end{itemize}
\end{prop}
\begin{proof}
Fix a vertex $\alpha\in\D$ and consider the four-holed sphere $S\backslash\alpha$.  Recall that $\C(S\backslash\alpha)$ is the Farey complex. We define a map $\theta_\alpha\colon\Lk(\alpha)\to\C(S\backslash\alpha)$, where $\Lk(\alpha)$ is the link of $\alpha$ in $\D$.
For a vertex $\alpha'\in\Lk(\alpha)$, $\theta_\alpha(\alpha')$ is the two-sided curve determined by $\alpha$ and $\alpha'$. It follows from Lemma \ref{Lem:3holedRP2} that $\theta_\alpha$ is a bijection on vertices, and we claim that it is a simplicial isomorphism. Indeed, note that for $\alpha',\alpha''\in\Lk(\alpha)$ we have $i(\alpha',\alpha'')=1\iff i(\theta_\alpha(\alpha'),\theta_\alpha(\alpha''))=2$. This proves (a). The other assertions are consequences of (a) and well-known properties of the Farey complex; namely $\dim\C(S\backslash\alpha)=2$ and  every edge of $\C(S\backslash\alpha)$ is contained in exactly two different triangles.
\end{proof}

Given a one-sided curve $\alpha_0$ we can construct infinitely many tetrahedra of $\D$ containing $\alpha_0$ as a vertex. Let $\{\beta_i\}_1^3$ be any Farey triangle of $\C(S\backslash\alpha_0)$ and, for $1\le i\le 3$, let $\alpha_i$ be the one-sided curve such that $\beta_i$ is determined by $\alpha_i$ and $\alpha_0$. Then
$\{\alpha_i\}_0^3$ is a tetrahedron of $\D$.

We define a ``dual'' graph $\T$ whose vertices are tetrahedra of $\D$. Two tetrahedra are connected by an edge in $\T$ if their intersection is a triangle. 
The following theorem follows from \cite[Theorem 3.1]{Sch}.
\begin{theorem}[Sharlemann]
$\T$ is a $4$-regular tree.
\end{theorem}

\begin{lemma}\label{L:Erigid} 
For any two tetrahedra $\{\alpha_i\}_{i=0}^3$ and $\{\alpha'_i\}_{i=0}^3$ of $\D$ there exists a unique $f\in\Mcg(S)$ such that
$f(\alpha_i)=\alpha'_i$ for $0\le i\le  3$.
\end{lemma}
\begin{proof}
For $1\le i\le 3$ let $\beta_i$ (resp. $\beta'_i$) be the two-sided curve determined by $\alpha_0$ and $\alpha_i$ (resp. $\alpha'_0$ and $\alpha'_i$). Note that $\{\beta_1,\beta_2,\beta_3\}$ and $\{\beta'_1,\beta'_2,\beta'_3\}$ are Farey triangles in $\C(S\backslash\alpha_0)$ and $\C(S\backslash\alpha'_0)$ respectively. By Lemma \ref{L:Farey_rigid} there exists a unique $f\in\Mcg(S)$ such that $f(\alpha_0)=\alpha'_0$ and $f(\beta_i)=\beta'_i$ for $1\le i\le 3$. Since $\alpha'_i$ is the unique vertex of $\C(S)$ different from $\alpha'_0$ and adjacent to $\beta'_i$, we have $f(\alpha_i)=\alpha'_i$ for $1\le i\le 3$.
\end{proof}

\begin{figure}[hbt]
 \begin{center}
 \includegraphics{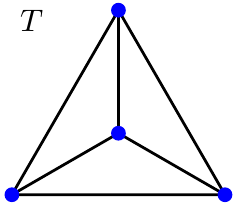} \hskip 1cm \includegraphics{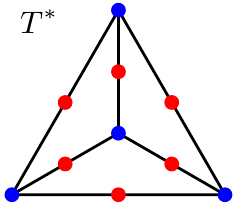}
\caption {A tetrahedron of $\D$ and the corresponding subgraph of $\C(S)$.}\label{Fig:TT}
\end{center}
\end{figure}

Let $T$ be a tetrahedron of $\D$. We denote by $T^\ast$ the full subcomplex of $\C(S)$ spanned by the four vertices of $T$ and the six two-sided vertices determined by the edges of $T$ (Figure \ref{Fig:TT}). The following proposition says that $T^\ast$ is rigid. It is thus an extension of the main result of \cite{IK}.

\begin{prop}\label{P:Trigid}
Suppose that $T$ is a tetrahedron $\D$ and $\phi\colon T^\ast\to\C(S)$ is a locally injective simplicial map. Then there exists a unique $f\in\Mcg(S)$ such that $\phi=f$ on $T^\ast$.
\end{prop}
\begin{proof}
First note that $\phi$ is injective because it is locally injective and 
 $T^\ast$ has diameter $2$.
Let $T=\{\alpha_i\}_{i=0}^3$. We claim that $\{\phi(\alpha_i)\}_{i=0}^3$ is a tetrahedron of $\D$. Indeed, for $1\le i\le 3$, $\phi(\alpha_i)$ is adjacent in $\C(S)$ to three different vertices, and hence it is one-sided, as two-sided vertices of $\C(S)$ have degree $2$. For $i\ne j$, the distance in $\C(S)$ between $\phi(\alpha_i)$ and $\phi(\alpha_j)$ is $2$, and hence $\phi(\alpha_i)$ and $\phi(\alpha_j)$ are adjacent in $\D$.

By Lemma \ref{L:Erigid} there exists a unique $f\in\Mcg(S)$ such that
$f(\alpha_i)=\phi(\alpha_i)$ for $0\le i\le  3$. Let $\beta$ be a two-sided 
vertex of $T^\ast$ determined by $\alpha_i$ and $\alpha_j$. Then $\phi(\beta)$ is adjacent to $\phi(\alpha_i)$ and $\phi(\alpha_j)$, and since such a curve is unique, $\phi(\beta)=f(\beta)$.
\end{proof}

We denote by $\T^0$ the vertex set of $\T$, that is the set of tetrahedra of $\D$.
Let $d_{\T}$ denote the path metric on $\T$.
We fix a reference tetrahedron $T_0$ %=\{\alpha_1,\epsilon_1,\delta_1,h_2(\epsilon_1)\}$ 
and define
\[\T^0_n=\{T\in\T^0 \mid d_\T(T,T_0)\le n\}.\]
In other words, $\T^0_n$ is the set of tetrahedra within distance at most $n$ from $T_0$ in the path metric on $\T$.

\medskip

\begin{proof}[Proof of Theorem \ref{T:main}.]
 Let $\cY_1=T_0^\ast$ and for $n\ge 1$:
 \[\cY_{n+1}=\bigcup_{T\in\T^0_n}T^\ast.\]
 We prove by induction that $\cY_n$ is rigid for all $n\ge 1$.
 By Proposition \ref{P:Trigid} $\cY_1$ is rigid. Assume that $\cY_n$ is rigid and let $\phi\colon\cY_{n+1}\to\C(S)$ be a locally injective simplicial map.
 Since $\cY_n$ is rigid, there exists a unique $f\in\Mcg(S)$ such that $f=\phi$ on $\cY_n$. Let $\phi'=f^{-1}\circ\phi$.

Let $T\in\T_{n+1}\backslash\T_n$. We need to show that $\phi'$ fixes every vertex of $T^\ast$. It suffices to show that $\phi'$ fixes every vertex of $T$ because then it also has to fix the two-sided vertices of $T^\ast$ determined by edges of $T$. The tetrahedron $T$ has a common face with some (unique) tetrahedron $T'\in\T_n$. Let $T=\{\alpha_0,\alpha_1,\alpha_2,\alpha_3\}$ and 
$T'=\{\alpha'_0,\alpha_1,\alpha_2,\alpha_3\}$.  Let $\beta$ (resp. $\beta'$) be the two-sided vertex of $T^\ast$ (resp. $(T')^\ast$) determined by $\alpha_0$ and $\alpha_1$ (resp.  $\alpha'_0$ and $\alpha_1$). By local injectivity of $\phi'$, $\phi'(\beta)\ne\phi'(\beta')=\beta'$, and hence also
$\phi'(\alpha_0)\ne\phi'(\alpha'_0)=\alpha'_0$. By Proposition  \ref{P:Trigid}, 
$\phi'(T)$ is  a tetrahedron different from $T'$ and having a common face with $T'$. Since such a tetrahedron is unique by (c) of Proposition \ref{Prop:D}, $\phi'(T)=T$ and $\phi'(\alpha_0)=\alpha_0$.
We have shown that $\phi'$ pointwise fixes $T^\ast$, and it follows that it pointwise fixes $\cY_{n+1}$. Hence $\phi=f$ on $\cY_{n+1}$. 

Since $\cY_n$ contains $T_0^\ast$ for all $n\ge 1$, it has trivial pointwise stabilizer in $\Mcg(S)$. Finally, it follows from the connectedness of $\T$ that $\bigcup_{n\ge 1}\cY_n=\C(S)$.
\end{proof}

\section{Coarse geometry}
In this section we consider $\C(S)$ and $\D^1$ as metric graphs with all edges of length $1$. We denote the metrics on these graphs by $d_\C$ and $d_\D$ respectively.

There is a natural piecewise-linear homeomorphism $\phi\colon\C(S)\to\D^1$ equal to the identity on one-sided vertices and which forgets the two-sided vertices. That is, if $\beta$ is the two-sided vertex of $\C(S)$ determined by $\alpha$ and $\alpha'$, then $\phi(\beta)$ is defined to be the midpoint of the edge of $\D$ connecting $\alpha$ and $\alpha'$. We have 
\[ d_\C(x,y)=2d_\D(\phi(x),\phi(y))\] for all $x,y\in\C(S)$. In particular, $\phi$ is a quasi-isometry.

Since $\mathcal{T}$ is a tree, every triangle of $\D$ is separating, i.e. the space obtained by removing a triangle from $\D$ has two connected components.
If $\Delta$ is a triangle of $\D$, and $x$ and $y$ are points lying in different connected components of $\D\backslash\Delta$, then we say that $\Delta$ separates $x$ from $y$.

\begin{lemma}\label{Lem:bottle}
Let $p$ be a vertex on a geodesic in $\D^1$ from  $x$ to $y$, such that $d_\D(p,x)\ge 1$ and $d_\D(p,y)\ge 1$. There exists a triangle $\Delta$ of $\D$ such that $p\in\Delta$ and $\Delta$ separates $x$ from $y$. 
\end{lemma}
\begin{proof}
Let $[x,y]$ be a geodesic in $\D^1$ from $x$ to $y$ containing $p$, and let $q$ be the vertex preceding $p$ on $[x,y]$. 
Let $(T_i)_0^n$ be any sequence of tetrahedra forming a geodesic in $\T$ such that  $q\in T_0$ and $y\in T_n$.  Note that $q\notin T_n$ since $d_\D(q,y)=1+d_\D(p,y)\ge 2$. Let $T_i$ be the first tetrahedron in this sequence such that $q\notin T_i$. Then  $\Delta=T_i\cap T_{i-1}$ is a triangle separating $q$ from $y$. The segment $[q,y]$ must pass through a vertex of  $\Delta$, which gives $p\in\Delta$. Finally notice that $\Delta$ separates $x$ from $y$, for otherwise $[x,y]$ could not contain $q$ (there would be shorter path from $x$ to $y$ avoiding $q$).
\end{proof}

\begin{theorem}
The curve graph $\C(S)$ is quasi-isometric to a simplicial tree.
\end{theorem}
\begin{proof}
Since $\C(S)$ is quasi-isometric to $\D^1$, it suffices to show that $\D^1$ is 
quasi-isometric to a simplicial tree. By \cite[Throem 4.6]{Manning}, this is equivalent to $\D^1$ satisfying the following Bottleneck Property:
There is some $L>0$ so that for all $x,y\in\D^1$ there is a midpoint $m=m(x,y)$ with $d(x,m)=d(y,m)=\frac{1}{2}d(x,y)$ and the property that any path from $x$ to $y$ must pass within less than $L$ of the point $m$.

Let $L>\frac{3}{2}$ and define $m=m(x,y)$ to be the midpoint of any geodesic from $x$ to $y$. Clearly we can assume $d_\D(x,m)\ge L$. Let $p$ be a vertex on a geodesic from $x$ to $y$ such that $d_\D(m,p)\le \frac{1}{2}$. By Lemma \ref{Lem:bottle}, there exists a triangle $\Delta$ separating $x$ from $y$ such that $p\in \Delta$. Any path from $x$ to $y$ must pass through $\Delta$, and hence within at most $\frac{3}{2}$ of the point $m$.
\end{proof}
By using the Bottleneck Property from the proof above we can determine the hyperbolicity constant of $\C(S)$.
%We close this section by recalling the definition of a hyperbolic space.
Recall that a geodesic metric space  $(X,d)$ is $\delta$-hyperbolic if, for any geodesic triangle $[x,y]\cup[x,z]\cup[y,z]$ and any $p\in[x,y]$ there exists some $q\in[x,z]\cup[y,z]$ with $d(p,q)\le\delta$. 
A triangle satisfying the condition above is called $\delta$-thin.
Masur and Minsky \cite{Mas-Min} proved that $\C(S)$ is hyperbolic for orientable $S$. Their result was extended to nonorientable surfaces by Bestvina-Fujiwara \cite{Best-Fuji} and Masur-Schleimer \cite{Mas-Sch}.
\begin{prop}
The graph $\C(S)$ is $3$-hyperbolic.
\end{prop}
\begin{proof}
First we prove that $\D^1$ is $\frac{3}{2}$-hyperbolic. Let $[x,y]\cup[x,z]\cup[y,z]$ be geodesic triangle in $\D^1$ and $p\in[x,y]$. Clearly we can assume $d_\D(x,p)\ge\frac{3}{2}$. Let $p'$ be a vertex on $[x,y]$ such that $d_\D(p,p')\le \frac{1}{2}$. By Lemma \ref{Lem:bottle}, there exists a triangle $\Delta$ separating $x$ from $y$ such that $p'\in \Delta$. 
It follows that $[x,z]\cup[y,z]$ has non-empty intersection with $\Delta$, and for any point $q$ in this intersection $d_\D(q,p)\le\frac{3}{2}$. 
%Notice that $T$ separates $z$ from either $x$ or $y$. It follows that either $[x,z]$ or $[y,z]$ must pass through $T$, and hence within less than $\frac{3}{2}$ of the point $p$, which proves that the triangle is $\frac{3}{2}$-thin.

To finish the proof we use the homeomorphism $\phi\colon\C(S)\to\D^1$. Observe that $\phi$ maps geodesics triangles to geodesic triangles and $d_\C(x,y)=2d_\D(\phi(x),\phi(y))$ for all $x,y\in\C(S)$. Since geodesic triangles in $\D^1$ are $\frac{3}{2}$-thin, geodesic triangles in $\C(S)$ are $3$-thin.
\end{proof}
%%%%%%%%%%%%%%%%%%%%%%%%%%%%%%%%%%%%%%%%%%%%%%%%%%%%%%%
%%%  Bibliography  %%%
%%%%%%%%%%%%%%%%%%%%%%%


\begin{thebibliography}{99}
%
\bibitem{AL1} J. Aramayona and C.J. Leininger, Finite rigid sets in curve complexes, J. Topol. Anal. 5 (2013), 183--203.
%
\bibitem{AL2} J. Aramayona and C.J. Leininger, Exhausting curve complexes by finite rigid sets, Pacific J. Math. 282 (2016), 257--283.
%
%
\bibitem{AK} F. Atalan and M. Korkmaz, Automorphisms of curve complexes on nonorientable surfaces, Groups Geom. Dyn. 8 (2014), 39--68.
%
%
\bibitem{Best-Fuji} M. Bestvina and K. Fujiwara, Quasi-homomorphisms on mapping class groups, {Glas. Mat. Ser. III} 42 (2007), 213--236.
%
%\bibitem{Chill}  Chillingworth, D.R.J. A finite set of generators for the homeotopy group of a non-orientablesurface. {\it Proc. Camb. Phil. Soc.} \textbf{65} (1969), 409--430.
%
\bibitem{FM} B. Farb and D. Margalit, A primer on mapping class groups.
Princeton Mathematical Series {49}. {Princeton University Press, Princeton,} 2012. 
%
%\bibitem{GP} J. Guaschi, D. Juan-Pineda, A survey of surface braid groups and the lower algebraic K-theory of their group rings. L. Ji, A. Papadopoulos and S.-T. Yau. Handbook of Group Actions, Volume II, 32,International Press of Boston Inc., pp.23-76, 2015.
%
\bibitem{Harvey} W. J. Harvey, Boundary structure of the modular group, in: Riemann
surfaces and related topics: Proc. 1978 Stony Brook Conf., Ann.
Math. Stud. 97, Princeton University Press (1981), 245--251.
%
\bibitem{IK}  S. Ilbira and M. Korkmaz, Finite rigid sets in curve complexes of non-orientable surfaces, arXiv:1810.07964. 
%
\bibitem{Irmak14} E. Irmak, On simplicial maps of the complexes of curves of nonorientable surfaces, Algebr. Geom. Topol. 14 (2014), 1153--1180.
%
\bibitem{Irmak19} E. Irmak, Exhausting curve complexes by finite rigid sets on nonorientable surfaces,  arXiv:1906.09913.
%
\bibitem{Ivanov-Aut} N. Ivanov, Automorphisms of complexes of curves and of Teichmuller spaces, {Int. Math. Res. Notices} {14} (1997), 651-666.
%
\bibitem{Kork-CC} M. Korkmaz, Automorphisms of complexes of curves on punctured spheres and on punctured tori, {Topology Appl.} {95} (1999), 85--111.
%
%\bibitem{K} M. Korkmaz, Mapping class groups of nonorientable surfaces, Geom. Dedicata 89 (2002), 109-133.
%
%\bibitem{Lick1} {Lickorish, W.B.R.} Homeomorphisms of non-orientable two-manifolds.
%\emph{Proc. Camb. Phil. Soc.} \textbf{59} (1963), 307--317.
%
\bibitem{Luo} F. Luo, Automorphisms of the complex of curves, {Topology} {39} (2002), 283--298.
%
\bibitem{Manning} J. Manning, Geometry of pseudocharacters, Geom. Topol. 9 (2005) 1147-1185.
%
\bibitem{Mas-Min} H. A. Masur and Y. N. Minsky, Geometry of the complex of
curves I: Hyprebolicty, {Invent. Math.} {138}~(1) (1999) 103--149.
%
\bibitem{Mas-Sch} H. A. Masur and S. Schleimer, The geometry of the disk complex, {J. Amer. Math. Soc.} 26 (2013),  1--62.
%\bibitem{PSz} Paris, L.;  Szepietowski, B. A presentation for the mapping class group of a nonorientable surface. \emph{Bull. Soc. Math. France} {\bf 143} (2015), 503--566.
%
\bibitem{Sch} M. Scharlemann, The complex of curves of a nonorientable surface, J.  London Math. Soc. 25 (1982), 171-184.
%
\bibitem{Shack} K. Shackleton, Combinatorial rigidity in curve complexes and mapping class groups, Pacific J. Math. 230 (2007), 217--232.
%
%\bibitem{Sze} Szepietowski, B.: A presentation for the mapping class group of a non-orientable surface from the action on the complex of curves. Osaka J. Math. 45, 283--326 (2008).
\end{thebibliography}
\end{document}